\documentclass[reqno]{amsart} 
\usepackage{amssymb}
\usepackage{amsmath}

\newtheorem{thm}{Theorem}
\newtheorem{prop}{Proposition}

\newtheorem{cor}{Corollary}
\newtheorem{rem}{Remark}

\newtheorem{lem}{Lemma}

\newcommand{\g}{\gamma} 

\newcommand{\G}{\Gamma}

\newcommand{\Cbb}{\mathbb{C}}

\newcommand{\lp}{\left(}
\newcommand{\rp}{\right)}
\newcommand{\lc}{\left\{}
\newcommand{\rc}{\right\}}
\newcommand{\lb}{\left[}
\newcommand{\rb}{\right]}

\newcommand{\RA}{\Rightarrow}

\begin{document}

\begin{flushright} OU-HET 892 \end{flushright}

\title{Summation formulae for the bilateral basic hypergeometric series ${}_1\psi_1 ( a; b; q, z )$}

\author{Hironori Mori} 
\address{Department of Physics, Graduate School of Science, Osaka University, Toyonaka, Osaka 560-0043, Japan}
\email{hiromori@het.phys.sci.osaka-u.ac.jp}

\author{Takeshi Morita} 
\address{Graduate School of Information Science and Technology, Osaka University, Toyonaka, Osaka 560-0043, Japan}
\email{t-morita@cr.math.sci.osaka-u.ac.jp}



\begin{abstract}
We give summation formulae for the bilateral basic hypergeometric series ${}_1\psi_1( a; b; q, z )$ through Ramanujan's summation formula, which are generalizations of nontrivial identities found in the physics of three-dimensional Abelian mirror symmetry on $\mathbf{R}P^2 \times S^1$. We also show the $q \to 1 - 0$ limit of our summation formulae.
\end{abstract}

\maketitle

\section{Introduction} \label{Intro}
In this paper, we give the following summation formulae for the bilateral series:
\begin{align} 
&
\frac{1}{2}
\left\{
\frac{\left( \frac{\gamma}{\alpha}; q \right)_\infty}{( \alpha\beta; q )_\infty}
\sum_{n \in \mathbb{Z}}
\frac{( \alpha \beta; q )_n}{( \frac{\gamma}{\alpha}; q )_n}
\left( \frac{\gamma w}{\alpha q^{1/2}} \right)^n
+
\frac{( \beta^2 \gamma; q )_\infty}{( \frac{1}{\beta}; q )_\infty}
\sum_{n \in \mathbb{Z}}
\frac{( \frac{1}{\beta}; q )_n}{( \beta^2 \gamma; q )_n}
\left( \frac{\gamma w}{\alpha q^{1/2}} \right)^n
\right\} \notag \\ 
&=
\frac
{\left( q^2, \frac{\g}{\alpha^2 \beta}, \frac{\alpha q^{3/2}}{\g w}, \alpha \beta w q^{1/2}; q^2 \right)_\infty}
{\left( q, \alpha^2 \beta^2, \frac{\g w}{\alpha q^{1/2}}, \frac{q^{1/2}}{\alpha \beta w}; q^2 \right)_\infty}, 
\label{main1}
\end{align}
and
\begin{align} 
&
\frac{1}{2}
\left\{
\frac{\left( \frac{\gamma}{\alpha}; q \right)_\infty}{( \alpha\beta; q )_\infty}
\sum_{n \in \mathbb{Z}}
\frac{( \alpha \beta; q )_n}{( \frac{\gamma}{\alpha}; q )_n}
\left( \frac{\gamma w}{\alpha q^{1/2}} \right)^n
-
\frac{( \beta^2 \gamma; q )_\infty}{( \frac{1}{\beta}; q )_\infty}
\sum_{n \in \mathbb{Z}}
\frac{( \frac{1}{\beta}; q )_n}{( \beta^2 \gamma; q )_n}
\left( \frac{\gamma w}{\alpha q^{1/2}} \right)^n
\right\} \notag \\ 
&=
\alpha \beta
\frac
{\left( q^2, \frac{\g}{\alpha^2 \beta}, \frac{\alpha q^{5/2}}{\g w}, \alpha \beta w q^{3/2}; q^2 \right)_\infty}
{\left( q, \alpha^2 \beta^2, \frac{\g w q^{1/2}}{\alpha}, \frac{q^{3/2}}{\alpha \beta w}; q^2 \right)_\infty}, 
\label{main2}
\end{align}
where $w \in \mathbb{C}^*$, and the parameters are constrained by $\alpha \beta^2 = - 1$ and $\beta \gamma = q$. Here, $( a; q )_n$ is the $q$-shifted factorial defined by
\begin{align*}
( a; q )_n := \lc
	\begin{aligned}
	& 1, && n = 0, \\ 
	& ( 1 - a ) ( 1 - a q ) \dots ( 1 - a q^{n - 1} ), && n \ge 1, \\ 
	& \lb ( 1 - a q^{-1} ) ( 1 - a q^{-2} ) \dots ( 1 - a q^n ) \rb^{-1}, && n \le - 1. 
	\end{aligned}
\right.
\end{align*}
Moreover, $( a; q )_\infty := \lim_{n \to \infty} ( a; q )_n$ and we use the shorthand notation
\begin{align*}
( a_1, a_2, \dots, a_m; q )_\infty := ( a_1; q )_\infty ( a_2; q )_\infty \dots ( a_m; q )_\infty.
\end{align*}
Surprisingly, the special case of the formulae \eqref{main1} and \eqref{main2} are originally shown in the context of physics \cite{MMT}. These are realized as the equality of a physical quantity called the superconformal index, which is calculated to get strong evidence for so-called Abelian mirror symmetry in three-dimensional supersymmetric quantum field theories.

In Section \ref{Main}, we provide the detailed proofs for the formulae \eqref{main1} and \eqref{main2}, which are the generalized versions of the formulae found in \cite{MMT}, based on the theta function of Jacobi and Ramanujan's summation formula
\begin{align} \label{rama}
\sum_{n \in \mathbb{Z}}
\frac{( a; q )_n}{( b; q )_n}
z^n
=
\frac{( q, b/a, a z, q/a z; q )_\infty}{( b, q/a, z, b/a z; q )_\infty}, \quad |b/a| < |z| < 1,
\end{align}
where the left-hand side is the bilateral basic hypergeometric series ${}_1\psi_1 ( a; b; q, z )$ (see Section \ref{Notation} for the explicit definition). This was given by S.~Ramanujan \cite{Ramanujan}.
Further contributions to summation formulae and transformations for the bilateral basic hypergeometric series with the base $q$ were given by Bailey \cite{Bailey, B1}, Slater \cite{Slater}, Jackson \cite{J1}, Jackson \cite{J2},   K.~R.~Vasuki and K.~R.~Rajanna \cite{VR}, and D.~D.~Somashekara, K.~N.~Murthy and S.~L.~Shalini \cite{SNS}. The specific point we should mention is that our summation formulae \eqref{main1} and \eqref{main2} connect the geometric series with two different bases $q$ and $q^2$. As an application, we also show that our formulae with taking the specific parameter combination reproduce the results obtained by the physical study of Abelian mirror symmetry on $\mathbf{R}P^2 \times S^1$ \cite{MMT} where $\mathbf{R}P^2$ is a real projective plane. This means that our generic formulae can be powerful tools for checking exactly the nontrivial statement in physics. 

In Section \ref{Limit}, we show the $q \to 1 - 0$ limit of our new formulae \eqref{main1} and \eqref{main2} from the viewpoint of connection problems on $q$-difference equations \cite{Birkhoff, Z0, M0}. The limit results in the relation among the gamma function, power functions and the bilateral hypergeometric series ${}_1H_1 ( a; b; z )$ with a suitable condition.
Other summation formulae for some bilateral hypergeometric series were studied by J.~Dougall \cite{Dougall} and W.~N.~Bailey \cite{Bailey}. For instance, Dougall derived the bilateral hypergeometric identity
\begin{align*}
{}_2H_2 ( a, b; c, d; 1 )
=
\frac{\Gamma ( 1 - a )\Gamma ( 1 - b ) \Gamma ( c ) \Gamma ( d ) \Gamma ( c + d - a - b - 1 )}{\Gamma ( c - a ) \Gamma ( c - b ) \Gamma ( d - a ) \Gamma ( d - b )},
\end{align*}
where $\Re ( c + d - a - b ) > 1$. Dougall also proved that a \textit{well-poised} series ${}_5H_5$ \cite{GR} could be evaluated at $w = 1$, and then we can obtain another summation formula for a well-poised series ${}_5H_5$ with $w = 1$ \cite{Slater}. In the $q \to 1 - 0$ limit of the formulae \eqref{main1} and \eqref{main2}, we reach to the formula for ${}_1H_1$ under a suitable condition shown in Theorem \ref{mainlim2}.

\section{Notation} \label{Notation}
In this section, we review basic notation.
The bilateral basic hypergeometric series with the base $q$ is given by
\begin{align} \label{bbh} 
{}_r\psi_s ( a_1, \dots, a_r; b_1, \dots, b_s; q, z ) 
:=
\sum_{n \in \mathbb{Z}}
\frac{( a_1, \dots, a_r; q )_n}{( b_1, \dots, b_s; q )_n}
\left\{ ( - 1 )^n q^{\frac{n ( n - 1 )}{2}} \right\}^{s - r} z^n.
\end{align}
The series \eqref{bbh} diverges for $z \not = 0$ if $s < r$ and converges for $|b_1 \dots b_s/a_1 \dots a_r| < |z| < 1$ if $r = s$ (see \cite{GR} for more details). We remark that the $q$-shifted factorial $( a; q )_n$ is the $q$-analogue of the shifted factorial
\begin{align*} 
( \alpha )_n
=
\alpha \{ \alpha + 1 \} \cdots \{ \alpha + ( n - 1 ) \},
\end{align*}
and the series \eqref{bbh} is the $q$-analogue of the bilateral hypergeometric function
\begin{align*}
{}_rH_s ( \alpha_1, \dots, \alpha_r; \beta_1, \dots, \beta_s; z )
:=
\sum_{n \in \mathbb{Z}}
\frac{( \alpha_1, \dots, \alpha_r )_n}{( \beta_1, \dots, \beta_s )_n}
z^n.
\end{align*}
By D'Alembert's ratio test, it can be checked that ${}_rH_r$ converges only for $|z| = 1$ \cite{Slater}, provided that $\Re ( \beta_1 + \dots + \beta_r - \alpha_1 - \dots - \alpha_r ) > 1$.

The $q$-gamma function $\Gamma_q ( z )$ is defined by
\begin{align} 
\Gamma_q ( z )
:=
\frac{( q; q )_\infty}{( q^z; q )_\infty}
( 1 - q )^{1 - z}, \qquad 0 < q < 1.
\end{align}
The $q \to 1 - 0$ limit of $\Gamma_q ( z )$ gives the gamma function \cite{GR}
\begin{align} 
\lim_{q \to 1 - 0}
\Gamma_q ( z ) = \Gamma ( z ).
\label{limgamma}
\end{align}
The theta function of Jacobi with the base $q$ is given by
\begin{align} \label{thetaJ} 
\theta_q ( z )
:=
\sum_{n \in \mathbb{Z}}
q^{\frac{n^2}{2}}
( - z )^n, \qquad \forall z \in \mathbb{C}^*.
\end{align}
Jacobi's triple product identity is 
\begin{align} 
\theta_q ( z )
=
( q, q^{1/2} z, q^{1/2}/z; q )_\infty.
\label{jtpi}
\end{align}
The theta function has the inversion formula
\begin{align} \label{inv} 
\theta_q ( z )
=
\theta_q \left( 1/z \right),
\end{align}
and satisfies the $q$-difference equation
\begin{align} 
\theta_q ( z q^k )
=
( - z )^{- k} q^{- \frac{k^2}{2}}
\theta_q ( z ).
\label{thetaperi}
\end{align}

In our study, the following proposition about the theta function \cite{Z1} is useful to consider the $q \to 1 - 0$ limit of our formulae in Section \ref{Limit}.
\begin{prop} 
For any $z \in \mathbb{C}^* ( - \pi < \arg z < \pi )$, we have
\begin{align}
\lim_{q \to 1 - 0}
\frac{\theta_q ( q^\beta z )}{\theta_q ( q^\alpha z )}
=
( - z )^{\alpha - \beta}.
\label{limt1}
\end{align}
\end{prop}
We also use the following limiting formula \cite{GR}:
\begin{align}
\lim_{q \to 1 - 0}
\frac{( z q^\alpha; q )_\infty}{( z; q )_\infty}
=
( 1 - z )^{- \alpha}, \qquad |z| < 1.
\label{limbin}
\end{align}

\section{Main theorem} \label{Main}
In this section, we show the formulae \eqref{main1} and \eqref{main2} as the summation formulae of the bilateral basic hypergeometric series ${}_1\psi_1 ( a; b; q, z )$ \cite{GR} by utilizing Ramanujan's summation formula.
\begin{thm} \label{mainth1} 
For any $w \in \mathbb{C}^*$, we have 
\begin{align}
&
\frac{1}{2}
\frac{( q; q^2 )_\infty}{( q^2; q^2 )_\infty}
\left\{
\frac{\left( \frac{\gamma}{\alpha}; q \right)_\infty}{( \alpha \beta; q )_\infty}
{}_1\psi_1 \left( \alpha \beta; \frac{\gamma}{\alpha}; q, \frac{\gamma w}{\alpha q^{1/2}} \right)
+
\frac{( \beta^2 \gamma; q )_\infty}{( \frac{1}{\beta}; q )_\infty}
{}_1\psi_1 \left( \frac{1}{\beta}; \beta^2 \gamma; q, \frac{\gamma w}{\alpha q^{1/2}} \right)
\right\} \notag \\ 
&=
\frac
{\left( \frac{\gamma}{\alpha^2 \beta}; q^2 \right)_\infty}
{( \alpha^2 \beta^2; q^2 )}
\frac
{\left( \frac{\alpha q^{1/2}}{\gamma w} q, \frac{\alpha \beta w}{q^{1/2}} q; q^2 \right)_\infty}
{\left( \frac{\gamma w}{\alpha q^{1/2}}, \frac{q^{1/2}}{\alpha \beta w}; q^2 \right)_\infty}, 
\label{m1}
\end{align}
where $\alpha \beta^2 = - 1$ and $\beta \gamma = q$.
\end{thm}

\begin{thm} \label{mainth2} 
For any $w \in \mathbb{C}^*$, we have 
\begin{align}
&
\frac{1}{2}
\frac{( q; q^2 )_\infty}{( q^2; q^2 )_\infty}
\left\{
\frac{\left( \frac{\gamma}{\alpha}; q \right)_\infty}{( \alpha \beta; q )_\infty}
{}_1\psi_1 \left( \alpha \beta; \frac{\gamma}{\alpha}; q, \frac{\gamma w}{\alpha q^{1/2}} \right)
-
\frac{( \beta^2 \gamma; q )_\infty}{( \frac{1}{\beta}; q )_\infty}
{}_1\psi_1 \left( \frac{1}{\beta}; \beta^2 \gamma; q, \frac{\gamma w}{\alpha q^{1/2}} \right)
\right\} \notag \\ 
&=
\alpha \beta
\frac
{\left( \frac{\gamma}{\alpha^2 \beta}; q^2 \right)_\infty}
{( \alpha^2 \beta^2; q^2 )}
\frac
{\left( \frac{\alpha q^{1/2}}{\gamma w} q^2, \frac{\alpha \beta w}{q^{1/2}} q^2; q^2 \right)_\infty}
{\left( \frac{\gamma w q^{1/2}}{\alpha}, \frac{q^{3/2}}{\alpha \beta w}; q^2 \right)_\infty}, 
\label{m2}
\end{align}
where $\alpha \beta^2 = - 1$ and $\beta \gamma = q$.
\end{thm}

We obtain the following corollaries immediately from Theorem \ref{mainth1} and \ref{mainth2}.
\begin{cor} 
For any $w \in \mathbb{C}^*$, we have 
\begin{align}
&
\frac{1}{2}
\left\{
\frac{\left( \frac{\gamma}{\alpha}; q \right)_\infty}{( \alpha \beta; q )_\infty}
{}_1\psi_1 \left( \alpha \beta; \frac{\gamma}{\alpha}; q, \frac{\gamma w}{\alpha q^{1/2}} \right)
+
\frac{( \beta^2 \gamma; q )_\infty}{( \frac{1}{\beta}; q )_\infty}
{}_1\psi_1 \left( \frac{1}{\beta}; \beta^2 \gamma; q, \frac{\gamma w}{\alpha q^{1/2}} \right)
\right\} \notag \\ 
&=
\frac{( q^2; q^2 )_\infty}{( q; q^2 )_\infty}
\frac
{\left( \frac{\gamma}{\alpha^2 \beta}; q^2 \right)_\infty}
{( \alpha^2 \beta^2; q^2 )}
\frac
{\left( \frac{\alpha q^{1/2}}{\gamma w} q, \frac{\alpha \beta w}{q^{1/2}} q; q^2 \right)_\infty}
{\left( \frac{\gamma w}{\alpha q^{1/2}}, \frac{q^{1/2}}{\alpha \beta w}; q^2 \right)_\infty}, 
\label{mcor1}
\end{align}
where $\alpha \beta^2 = - 1$ and $\beta \gamma = q$. 
\end{cor}

\begin{cor} 
For any $w \in \mathbb{C}^*$, we have 
\begin{align}
&
\frac{1}{2}
\left\{
\frac{\left( \frac{\gamma}{\alpha}; q \right)_\infty}{( \alpha \beta; q )_\infty}
{}_1\psi_1 \left( \alpha \beta; \frac{\gamma}{\alpha}; q, \frac{\gamma w}{\alpha q^{1/2}} \right)
-
\frac{( \beta^2 \gamma; q )_\infty}{( \frac{1}{\beta}; q )_\infty}
{}_1\psi_1 \left( \frac{1}{\beta}; \beta^2 \gamma; q, \frac{\gamma w}{\alpha q^{1/2}} \right)
\right\} \notag \\ 
&=
\alpha \beta
\frac{( q^2; q^2 )_\infty}{( q; q^2 )_\infty}
\frac
{\left( \frac{\gamma}{\alpha^2 \beta}; q^2 \right)_\infty}
{( \alpha^2 \beta^2; q^2 )}
\frac
{\left( \frac{\alpha q^{1/2}}{\gamma w} q^2, \frac{\alpha \beta w}{q^{1/2}} q^2; q^2 \right)_\infty}
{\left( \frac{\gamma w q^{1/2}}{\alpha}, \frac{q^{3/2}}{\alpha \beta w}; q^2 \right)_\infty}, 
\label{mcor2}
\end{align}
where $\alpha \beta^2 = - 1$ and $\beta \gamma = q$. 
\end{cor}
We consider the $q \to 1 - 0$ limit of our formulae \eqref{mcor1} and \eqref{mcor2} in Section \ref{Limit}.

\subsection{Ramanujan's summation formula}
We review Ramanujan's sum for ${}_1\psi_1 ( a; b; q, z )$ to establish our main theorems. 
\begin{thm}[Ramanujan's sum for ${}_1\psi_1( a; b; q, z )$] \label{Ram} 
For any $z \in \mathbb{C}$, we have
\begin{align} \label{rsum11}
 {}_1\psi_1 ( a; b; q, z )
 =
 \frac{( q, b/a, a z, q/a z; q )_\infty}{( b, q/a, z, b/a z; q )_\infty}
\end{align}
with $|b/a| < |z| < |1|$.
\end{thm}
There are a number of proofs of the summation formula \eqref{rsum11} in the literature. The first published proof of the summation formula \eqref{rsum11} was given by W.~Hahn \cite{Hahn} and M.~Jackson \cite{J2}. We can find other proofs by G.~E.~Andrews \cite{A1,A2}, Askey \cite{As1}, and M.~E.~H.~Ismail \cite{Ismail}. Ramanujan's summation formula \eqref{rsum11} is considered as the ``bilateral extension'' of the $q$-binomial theorem \cite{GR} 
\begin{align} \label{qbinomial} 
\sum_{n \ge 0}
\frac{( a; q )_n}{( q; q )_n}
z^n
=
\frac{( a z; q )_\infty}{( z; q )_\infty}
\end{align}
for $|z| < 1$. The $q$-binomial theorem \eqref{qbinomial} was derived by Cauchy \cite{C1}, Heine \cite{H1}, and many mathematicians. We remark that the identity \eqref{rsum11} is also thought of as the $q$-analogue of the ``bilateral binomial theorem'' \cite{Mbi} found by M.~E.~Horn \cite{Horn}
\begin{align*} 
{}_1H_1 ( a; c; z )
=
\frac{( 1 - z )^{c - a - 1}}{( - z )^{c - 1}}
\frac{\Gamma ( 1 - a ) \Gamma ( c )}{\Gamma ( c - a )},
\end{align*}
where $a$ and $c$ are complex numbers such that $\Re ( c - a ) > 1$, and $z$ is a complex number with $|z| = 1$ and $z \neq 1$.

For later use, let us write the application of Ramanujan's summation formula \eqref{rsum11}. We have the following relations by Theorem \ref{Ram}.
\begin{cor} \label{l1} 
For any $w \in \mathbb{C}^*$ (provided that $|\gamma w/( \alpha q^{1/2} )| < 1$), we have
\begin{align}
{}_1\psi_1 \left( \alpha \beta; \frac{\gamma}{\alpha}; q, \frac{\gamma w}{\alpha q^{1/2}} \right)
&=
\frac
{\left( q, \frac{\gamma}{\alpha^2 \beta}, \frac{\beta \gamma w}{q^{1/2}}, \frac{q^{3/2}}{\beta \gamma w}; q \right)_\infty}
{\left( \frac{\gamma}{\alpha}, \frac{q}{\alpha \beta}, \frac{\gamma w}{\alpha q^{1/2}}, \frac{q^{1/2}}{\alpha \beta w}; q \right)_\infty}, \\ 
{}_1\psi_1 \left( \frac{1}{\beta}; \beta^2 \gamma; q, \frac{\gamma w}{\alpha q^{1/2}} \right)
&=
\frac
{\left( q, \beta^3 \g, \frac{\gamma w}{\alpha \beta q^{1/2}}, \frac{\alpha \beta q^{3/2}}{\gamma w}; q \right)_\infty}
{\left( \beta^2 \gamma, q \beta, \frac{\gamma w}{\alpha q^{1/2}}, \frac{\alpha \beta^{3} q^{1/2}}{w}; q \right)_\infty}
=
\frac
{\left( q, \frac{\g}{\alpha^2 \beta}, \frac{\gamma w}{\alpha \beta q^{1/2}}, \frac{\alpha \beta q^{3/2}}{\gamma w}; q \right)_\infty}
{\left( \beta^2 \gamma, q \beta, \frac{\gamma w}{\alpha q^{1/2}}, \frac{q^{1/2}}{\alpha \beta w}; q \right)_\infty}. 
\end{align}
\end{cor}

\begin{rem} 
Recall that there exist relations $\alpha \beta^2 = - 1$ and $\beta \gamma = q$. 
\end{rem}

\subsection{Relations between the theta functions}
Before showing the precise proofs of our main theorems, in order to make them simpler, we prepare several relations for the theta functions with the same or the different bases. First of all, we obtain the following relations by using the inversion formula of the theta function \eqref{inv}:
\begin{lem} \label{lem1} 
For any $\xi, \eta \in \mathbb{C}^*$, we have
\begin{align*}
\theta_{q^2} ( \xi \eta )
&=
\frac{1}{2}
\left\{
\theta_{q^2} ( \xi \eta ) + \theta_{q^2} \left( \frac{1}{\xi \eta} \right)
\right\}, \\ 
\theta_{q^2} \left( \frac{\xi}{\eta} \right)
&=
\frac{1}{2}
\left\{
\theta_{q^2} \left( \frac{\xi}{\eta} \right) + \theta_{q^2} \left( \frac{\eta}{\xi} \right)
\right\}. 
\end{align*}
\end{lem}

\begin{cor} \label{co3} 
If we put $\xi = \alpha \beta/q^{1/2}$ and $\eta = 1/w$ in Lemma \ref{lem1}, we obtain
\begin{align*}
\frac
{1}
{\theta_{q^2} \left( \frac{\alpha \beta}{q^{1/2} w} \right)
\theta_{q^2} \left( \frac{\alpha \beta w}{q^{1/2}} \right)}
=
\frac
{4}
{\left\{
\theta_{q^2} \left( \frac{\alpha \beta}{q^{1/2} w} \right) + \theta_{q^2} \left( \frac{q^{1/2} w}{\alpha \beta} \right)
\right\}
\left\{
\theta_{q^2} \left( \frac{\alpha \beta w}{q^{1/2}} \right) + \theta_{q^2} \left( \frac{q^{1/2}}{\alpha \beta w} \right)
\right\}}. 
\end{align*}
\end{cor}

\begin{cor} \label{corl3} 
If we put $\xi = \alpha \beta q/q^{1/2}$ and $\eta = 1/w$ in Lemma \ref{lem1}, we obtain
\begin{align*}
\frac
{1}
{\theta_{q^2} \left( \frac{\alpha \beta}{q^{1/2} w} q \right)
\theta_{q^2} \left( \frac{\alpha \beta w}{q^{1/2}} q \right)}
=
\frac
{4}
{\left\{
\theta_{q^2} \left( \frac{\alpha \beta}{q^{1/2} w} q \right)
+
\theta_{q^2} \left( \frac{q^{1/2} w}{\alpha \beta q} \right)
\right\}
\left\{
\theta_{q^2} \left( \frac{\alpha \beta w}{q^{1/2}} q \right)
+
\theta_{q^2} \left( \frac{q^{1/2}}{\alpha \beta q w} \right)
\right\}}.
\end{align*}
\end{cor}

We can derive the product formula of the theta functions with the different bases.
\begin{prop} \label{p1} 
For any $\xi$, $\eta \in \mathbb{C}^*$, we have
\begin{align}
\theta_{q^2} \left( \frac{\xi}{\eta} \right) \theta_{q^2} ( \xi \eta )
=
\theta_q ( \xi ) \theta_q ( - \eta ).
\label{ptp1}
\end{align}
\end{prop}

\begin{proof}
We put $\xi = e^{\pi i x}$ and $\eta = e^{\pi i y}$ in the definition of the theta function \eqref{thetaJ} on the left-hand side of \eqref{ptp1},
\begin{align*}
&
\theta_{q^2} \left( \frac{\xi}{\eta} \right)
\theta_{q^2} ( \xi \eta ) \\ 
&=
\left(
\sum_{n \in \mathbb{Z}}
\exp \left\{ \pi i n + \pi i n^2 \tau + \pi i n ( x - y ) \right\}
\right)
\left(
\sum_{m \in \mathbb{Z}}
\exp \left\{ \pi i m + \pi i m^2 \tau + \pi i m ( x + y ) \right\}
\right) \\ 
&=
\sum_{n \in \mathbb{Z}} \sum_{m \in \mathbb{Z}}
\exp \left\{ \pi i ( n + m ) + \pi i ( n^2 + m^2 ) \tau + \pi i n ( x - y ) + \pi i m ( x + y ) \right\} \\ 
&=
\sum_{n \in \mathbb{Z}} \sum_{m \in \mathbb{Z}}
\exp \left\{ \frac{\pi i}{2} ( n + m )^2 \tau + \frac{\pi i}{2} ( n - m )^2 \tau + \pi i ( n + m ) x - \pi i ( n - m ) y + \pi i ( n + m ) \right\} \\ 
&=
\sum_{N \in \mathbb{Z}}
\exp \left\{ \frac{\pi i}{2} N^2 \tau + \pi i N x + \pi i N \right\}
\sum_{ - M \in \mathbb{Z}}
\exp \left\{ \frac{\pi i}{2} M^2 \tau + \pi i M y \right\} \\ 
&=
\theta_q ( \xi )
\theta_q ( - \eta ),
\end{align*}
provided that $N := n + m$ and $M := n - m$. Therefore, we obtain the identity \eqref{ptp1}. 
\end{proof}
By Proposition \ref{p1} and the inversion formula \eqref{inv}, we have the following relations:
\begin{cor} \label{co4} 
For the specific parameter combinations of $\xi$ and $\eta$, the identity \eqref{ptp1} can be expressed as
\begin{align*}
\xi = \frac{\alpha \beta}{q^{1/2}},\
\eta = w
&\RA
\theta_{q^2} \left( \frac{\alpha \beta}{q^{1/2} w} \right)
\theta_{q^2} \left( \frac{\alpha \beta w}{q^{1/2}} \right)
=
\theta_q \left( \frac{\alpha \beta}{q^{1/2}} \right)
\theta_q ( - w ), \\ 
\xi = \frac{1}{w},\
\eta = \frac{q^{1/2}}{\alpha \beta}
&\RA
\theta_{q^2} \left( \frac{\alpha \beta}{q^{1/2} w} \right)
\theta_{q^2} \left( \frac{q^{1/2}}{\alpha \beta w} \right)
=
\theta_q \left( \frac{1}{w} \right)
\theta_q \left( - \frac{q^{1/2}}{\alpha \beta} \right)
=
\theta_q \left( w \right)
\theta_q \left( - \frac{q^{1/2}}{\alpha \beta} \right), \\ 
\xi = w,\
\eta = \frac{\alpha \beta}{q^{1/2}}
&\RA
\theta_{q^2} \left( \frac{q^{1/2} w}{\alpha \beta} \right)
\theta_{q^2} \left( \frac{\alpha \beta w}{q^{1/2}} \right)
=
\theta_q ( w )
\theta_q \left( - \frac{\alpha \beta}{q^{1/2}} \right)
=
\theta_q \left( w \right)
\theta_q \left( - \frac{q^{1/2}}{\alpha \beta} \right), \\ 
\xi = \frac{q^{1/2}}{\alpha \beta},\
\eta = \frac{1}{w}
&\RA
\theta_{q^2} \left( \frac{q^{1/2} w}{\alpha \beta} \right)
\theta_{q^2} \left( \frac{q^{1/2}}{\alpha \beta w} \right)
=
\theta_q \left( \frac{q^{1/2}}{\alpha \beta} \right)
\theta_q \left( - \frac{1}{w} \right)
=
\theta_q \left( \frac{\alpha \beta}{q^{1/2}} \right)
\theta_q ( - w ). 
\end{align*}
\end{cor}
We apply the relations in Corollary \ref{co4} to the equality in Corollary \ref{co3} in order to change the base of the theta functions.
\begin{cor} \label{co5} 
For any $\alpha$, $\beta$ and $w \in \mathbb{C}^*$, we have
\begin{align*}
\frac
{1}
{\theta_{q^2} \left( \frac{\alpha \beta}{q^{1/2} w} \right)
\theta_{q^2} \left( \frac{\alpha \beta w}{q^{1/2}} \right)} 
=
\frac
{2}
{\theta_q \left( \frac{\alpha \beta}{q^{1/2}} \right)
\theta_q \left( \frac{w}{\alpha \beta^2} \right)
+
\theta_q \left( w \right)
\theta_q \left( \frac{1}{\beta q^{1/2}} \right)}.
\end{align*}
\end{cor}

\begin{proof} 
Bringing Corollary \ref{co3} and \ref{co4} together leads to 
\begin{align*}
&
\frac
{1}
{\theta_{q^2} \left( \frac{\alpha \beta}{q^{1/2} w} \right)
\theta_{q^2} \left( \frac{\alpha \beta w}{q^{1/2}} \right)} \\ 
&=
\frac
{4}
{\left\{
\theta_{q^2} \left( \frac{\alpha \beta}{q^{1/2} w} \right)
+
\theta_{q^2} \left( \frac{q^{1/2} w}{\alpha \beta} \right)
\right\}
\left\{
\theta_{q^2} \left( \frac{\alpha \beta w}{q^{1/2}} \right)
+
\theta_{q^2} \left( \frac{q^{1/2}}{\alpha \beta w} \right)
\right\}} \\ 
&=
\frac
{4}
{\theta_{q^2} \left( \frac{\alpha \beta}{q^{1/2} w} \right)
\theta_{q^2} \left( \frac{\alpha \beta w}{q^{1/2}} \right)
+
\theta_{q^2} \left( \frac{\alpha \beta}{q^{1/2} w} \right)
\theta_{q^2} \left( \frac{q^{1/2}}{\alpha \beta w} \right)
+
\theta_{q^2} \left( \frac{q^{1/2} w}{\alpha \beta} \right)
\theta_{q^2} \left( \frac{\alpha \beta w}{q^{1/2}} \right)  
+
\theta_{q^2} \left( \frac{q^{1/2} w}{\alpha \beta} \right)
\theta_{q^2} \left( \frac{q^{1/2}}{\alpha \beta w} \right)} \\ 
&=
\frac
{2}
{\theta_q \left( \frac{\alpha \beta}{q^{1/2}} \right)
\theta_q \left( - w \right)
+
\theta_q \left( w \right)
\theta_q \left( - \frac{q^{1/2}}{\alpha \beta} \right)} \\ 
&=
\frac
{2}
{\theta_q \left( \frac{\alpha \beta}{q^{1/2}} \right)
\theta_q \left( \frac{w}{\alpha \beta^2} \right)
+
\theta_q \left( w \right)
\theta_q \left( \beta q^{1/2} \right)} \\ 
&=
\frac
{2}
{\theta_q \left( \frac{\alpha \beta}{q^{1/2}} \right)
\theta_q \left( \frac{w}{\alpha \beta^2} \right)
+
\theta_q \left( w \right)
\theta_q \left( \frac{1}{\beta q^{1/2}} \right)}, 
\end{align*}
where we used the condition $\alpha \beta^2 = - 1$ and the inversion formula \eqref{inv} in the last two steps.
\end{proof}

Moreover, we need other product relations connecting the theta functions with different bases derived from Proposition \ref{p1} and the inversion formula \eqref{inv}.
\begin{cor} \label{corl4} 
For the specific parameter combinations of $\xi$ and $\eta$, the identity \eqref{ptp1} can be expressed as
\begin{align*}
\xi = \frac{\alpha \beta}{q^{1/2}} q,\
\eta = w
&\RA
\theta_{q^2} \left( \frac{\alpha \beta}{q^{1/2} w} q \right)
\theta_{q^2} \left( \frac{\alpha \beta w}{q^{1/2}} q \right)
=
\theta_q \left( \frac{\alpha \beta}{q^{1/2}} q \right)
\theta_q ( - w ), \\ 
\xi = \frac{1}{w},\
\eta = \frac{q^{1/2}}{\alpha \beta q}
&\RA
\theta_{q^2} \left( \frac{\alpha \beta}{q^{1/2} w} q \right)
\theta_{q^2} \left( \frac{q^{1/2}}{\alpha \beta q w} \right)
=
\theta_q \left( \frac{1}{w} \right)
\theta_q \left( - \frac{q^{1/2}}{\alpha \beta q} \right)
=
\theta_q \left( w \right)
\theta_q \left( - \frac{\alpha \beta}{q^{1/2}} q \right), \\ 
\xi = w,\
\eta = \frac{\alpha \beta}{q^{1/2}} q
&\RA
\theta_{q^2} \left( \frac{q^{1/2} w}{\alpha \beta q} \right)
\theta_{q^2} \left( \frac{\alpha \beta w}{q^{1/2}} q \right)
=
\theta_q ( w )
\theta_q \left( - \frac{\alpha \beta}{q^{1/2}} q \right), \\ 
\xi = \frac{q^{1/2}}{\alpha \beta q},\
\eta = \frac{1}{w}
&\RA
\theta_{q^2} \left( \frac{q^{1/2} w}{\alpha \beta q} \right)
\theta_{q^2} \left( \frac{q^{1/2}}{\alpha \beta q w} \right)
=
\theta_q \left( \frac{q^{1/2}}{\alpha \beta q} \right)
\theta_q \left( - \frac{1}{w} \right)
=
\theta_q \left( \frac{\alpha \beta}{q^{1/2}} q \right)
\theta_q ( - w ). 
\end{align*}
\end{cor}
Then, these relations in Corollary \ref{corl4} can be used to show another product-to-sum identity of the theta function.
\begin{cor} \label{corol1} 
For any $\alpha$, $\beta$ and $w \in \mathbb{C}^*$, we have
\begin{align*}
\frac
{1}
{\theta_{q^2} \left( \frac{\alpha \beta}{q^{1/2} w} q \right)
\theta_{q^2} \left( \frac{\alpha \beta w}{q^{1/2}} q \right)} 
=
\frac
{2 \alpha \beta}
{-
\theta_q \left( \frac{\alpha \beta}{q^{1/2}} \right)
\theta_q \left( \frac{w}{\alpha \beta^2} \right)
+
\theta_q \left( w \right)
\theta_q \left( \frac{1}{\beta q^{1/2}} \right)}.
\end{align*}
\end{cor}

\begin{proof} 
As done in Corollary \ref{co5}, bringing Corollary \ref{corl3} and \ref{corl4} together leads to 
\begin{align*}
\frac
{1}
{\theta_{q^2} \left( \frac{\alpha \beta}{q^{1/2} w} q \right)
\theta_{q^2} \left( \frac{\alpha \beta w}{q^{1/2}} q \right)}
&=
\frac
{4}
{\left\{
\theta_{q^2} \left( \frac{\alpha \beta}{q^{1/2} w} q \right)
+
\theta_{q^2} \left( \frac{q^{1/2} w}{\alpha \beta q} \right)
\right\}
\left\{
\theta_{q^2} \left( \frac{\alpha \beta w}{q^{1/2}} q \right)
+
\theta_{q^2} \left( \frac{q^{1/2}}{\alpha \beta q w} \right)
\right\}} \\ 
&=
\frac
{2}
{\theta_q \left( \frac{\alpha \beta}{q^{1/2}} q \right)
\theta_q \left( - w \right)
+
\theta_q \left( w \right)
\theta_q \left( - \frac{\alpha \beta}{q^{1/2}} q \right)}. 
\end{align*}
By applying the $q$-difference equation of the theta function \eqref{thetaperi}, 
\begin{align*}
&
\frac
{2}
{\theta_q \left( \frac{\alpha \beta}{q^{1/2}} q \right)
\theta_q \left( - w \right)
+
\theta_q \left( w \right)
\theta_q \left( - \frac{\alpha \beta}{q^{1/2}} q \right)} \\ 
&=
\frac
{2}
{\lp - \frac{q^{1/2}}{\alpha \beta} q^{- 1/2} \rp
\cdot
\theta_q \left( \frac{\alpha \beta}{q^{1/2}} \right)
\theta_q \left( - w \right)
+
\lp \frac{q^{1/2}}{\alpha \beta} q^{- 1/2} \rp
\cdot
\theta_q \left( w \right)
\theta_q \left( - \frac{\alpha \beta}{q^{1/2}} \right)} \\ 
&=
\frac
{2 \alpha \beta}
{-
\theta_q \left( \frac{\alpha \beta}{q^{1/2}} \right)
\theta_q \left( - w \right)
+
\theta_q \left( w \right)
\theta_q \left( - \frac{\alpha \beta}{q^{1/2}} \right)} \\ 
&=
\frac
{2 \alpha \beta}
{-
\theta_q \left( \frac{\alpha \beta}{q^{1/2}} \right)
\theta_q \left( \frac{w}{\alpha \beta^2} \right)
+
\theta_q \left( w \right)
\theta_q \left( \frac{1}{\beta q^{1/2}} \right)}, 
\end{align*}
where we used the condition $\alpha \beta^2 = - 1$ in the last equality.
\end{proof}

\subsection{Proof of Theorem \ref{mainth1}}
In this subsection, we give the exact proof of Theorem \ref{mainth1} based on the relations for the theta functions we showed above. We aim to translate the left-hand side of the relation \eqref{m1} into its right-hand side. By Corollary \ref{l1}, we can rewrite the left-hand side of \eqref{m1} as follows:
\begin{align} \label{proof1} 
&
\frac{1}{2}
\frac{( q; q^2 )_\infty}{( q^2; q^2 )_\infty}
\left\{
\frac{\left( \frac{\gamma}{\alpha}; q \right)_\infty}{( \alpha \beta; q )_\infty}
{}_1\psi_1 \left( \alpha \beta; \frac{\gamma}{\alpha}; q, \frac{\gamma w}{\alpha q^{1/2}} \right)
+
\frac{( \beta^2 \gamma; q )_\infty}{( \frac{1}{\beta}; q )_\infty}
{}_1\psi_1 \left( \frac{1}{\beta}; \beta^2 \gamma; q, \frac{\gamma w}{\alpha q^{1/2}} \right)
\right\} \notag \\ 
&=
\frac{1}{2}
\frac{( q; q^2 )_\infty}{( q^2; q^2 )_\infty}
\left\{
\frac{\left( \frac{\gamma}{\alpha}; q \right)_\infty}{(\alpha \beta; q )_\infty}
\frac
{\left( q, \frac{\gamma}{\alpha^2 \beta}, \frac{\beta \gamma w}{q^{1/2}}, \frac{q^{3/2}}{\beta \gamma w} \right)_\infty}
{\left( \frac{\gamma}{\alpha}, \frac{q}{\alpha \beta}, \frac{\gamma w}{\alpha q^{1/2}}, \frac{q^{1/2}}{\alpha \beta w}; q \right)_\infty}
+
\frac{( \beta^2 \gamma; q )_\infty}{( \frac{1}{\beta}; q )_\infty}
\frac
{\left( q, \frac{\gamma}{\alpha^2 \beta}, \frac{\gamma w}{\alpha \beta q^{1/2}}, \frac{\alpha \beta q^{3/2}}{\gamma w}; q \right)_\infty}
{\left( \beta^2 \gamma, \beta q, \frac{\gamma w}{\alpha q^{1/2}}, \frac{q^{1/2}}{\alpha \beta w}; q \right)_\infty}
\right\} \notag \\ 
&=
\frac{1}{2}
\frac{( q; q^2 )_\infty}{( q^2; q^2 )_\infty}
\frac{\left( q, \frac{\gamma}{\alpha^2 \beta}; q \right)_\infty}{\left( \frac{\gamma w}{\alpha q^{1/2}}, \frac{q^{1/2}}{\alpha \beta w}; q \right)_\infty}
\left\{
\frac{\left( \frac{\beta \gamma w}{q^{1/2}}, \frac{q^{3/2}}{\beta \gamma w}; q \right)_\infty}{\left( \alpha \beta, \frac{q}{\alpha \beta}; q \right)_\infty}
+
\frac
{\left( \frac{\gamma w}{\alpha \beta q^{1/2}}, \frac{\alpha \beta q^{3/2}}{\gamma w}; q \right)_\infty}
{\left( \frac{1}{\beta}, \beta q; q \right)_\infty}
\right\} \notag \\ 
&=
\frac{1}{2}
\frac{( q; q^2 )_\infty}{ ( q^2; q^2 )_\infty}
\frac
{\left( q, \frac{\gamma}{\alpha^2\beta}; q \right)_\infty}
{\left( \frac{\gamma w}{\alpha q^{1/2}}, \frac{q^{1/2}}{\alpha \beta w}; q \right)_\infty}
\frac
{1}
{\left( \alpha \beta, \frac{q}{\alpha \beta}, \frac{1}{\beta}, \beta q; q \right)_\infty} \notag \\ 
&\hspace{1em} \times 
\left\{
\left( \frac{\beta \gamma w}{q^{1/2}}, \frac{q^{3/2}}{\beta \gamma w}, \frac{1}{\beta}, \beta q; q \right)_\infty
+
\left( \frac{\gamma w}{\alpha \beta q^{1/2}}, \frac{\alpha \beta q^{3/2}}{\gamma w}, \alpha \beta, \frac{q}{\alpha\beta}; q \right)_\infty
\right\}. 
\end{align}
Now, we focus on the overall factor in the first line of \eqref{proof1}. By using the relations $( a^2; q^2 )_\infty = ( a, - a; q )_\infty$ and $( a; q )_\infty = ( a, a q; q^2 )_\infty$ in addition to $\alpha \beta^2 = - 1$, this part is rewritten as
\begin{align} 
&
\frac{1}{2}
\frac{( q; q^2 )_\infty}{( q^2; q^2 )_\infty}
\frac{\left( q, \frac{\gamma}{\alpha^2 \beta}; q \right)_\infty}{\left( \frac{\gamma w}{\alpha q^{1/2}}, \frac{q^{1/2}}{\alpha \beta w}; q \right)_\infty}
\frac{1}{\left( \alpha \beta, \frac{q}{\alpha \beta}, \frac{1}{\beta}, \beta q; q \right)_\infty} \notag \\ 
&=
\frac{1}{2}
\frac{( q; q^2 )_\infty}{( q^2; q^2 )_\infty}
\frac{\left( q, \frac{\gamma}{\alpha^2 \beta}; q \right)_\infty}{\left( \frac{\gamma w}{\alpha q^{1/2}}, \frac{q^{1/2}}{\alpha \beta w}; q \right)_\infty}
\frac{1}{\left( \alpha \beta, \frac{q}{\alpha \beta}, - \alpha \beta, - \frac{q}{\alpha \beta}; q \right)_\infty} \notag \\ 
&=
\frac{1}{2}
\frac{( q; q^2 )_\infty}{( q^2; q^2 )_\infty}
\frac
{\left( q, q^2, \frac{\gamma}{\alpha^2 \beta}, \frac{\gamma}{\alpha^2 \beta} q; q^2 \right)_\infty}
{\left( \frac{\gamma w}{\alpha q^{1/2}}, \frac{q^{1/2}}{\alpha \beta w}; q \right)_\infty}
\frac{1}{\left( \alpha^2 \beta^2, \frac{q^2}{\alpha^2 \beta^2}; q^2 \right)_\infty}. 
\end{align}
Furthermore, the condition $\beta \g = q$ can simplify this part,
\begin{align}
&
\frac{1}{2}
\frac{( q; q^2 )_\infty}{( q^2; q^2 )_\infty}
\frac
{\left( q, q^2, \frac{\gamma}{\alpha^2 \beta}, \frac{\gamma}{\alpha^2 \beta} q; q^2 \right)_\infty}
{\left( \frac{\gamma w}{\alpha q^{1/2}}, \frac{q^{1/2}}{\alpha \beta w}; q \right)_\infty}
\frac{1}{\left( \alpha^2 \beta^2, \frac{q^2}{\alpha^2 \beta^2}; q^2 \right)_\infty} \notag \\ 
&=
\frac{1}{2}
\frac{( q; q^2 )_\infty}{( q^2; q^2 )_\infty}
\frac
{\left( q, q^2, \frac{\gamma}{\alpha^2 \beta}, \frac{q^2}{\alpha^2 \beta^2}; q^2 \right)_\infty}
{\left( \frac{\gamma w}{\alpha q^{1/2}}, \frac{q^{1/2}}{\alpha \beta w}; q \right)_\infty}
\frac{1}{\left( \alpha^2 \beta^2, \frac{q^2}{\alpha^2 \beta^2}; q^2 \right)_\infty} \notag \\ 
&=
\frac{1}{2}
( q; q^2 )_\infty
\frac{\left( q, \frac{\gamma}{\alpha^2 \beta}; q^2 \right)_\infty}{\left( \frac{\gamma w}{\alpha q^{1/2}}, \frac{q^{1/2}}{\alpha \beta w}; q \right)_\infty}
\frac{1}{\left( \alpha^2 \beta^2; q^2 \right)_\infty} \notag \\ 
&=
\frac{1}{2}
( q; q^2 )_\infty
\frac
{\left( q, \frac{\gamma}{\alpha^2 \beta}; q^2 \right)_\infty}
{\left( \frac{\gamma w}{\alpha q^{1/2}}, \frac{\gamma w}{\alpha q^{1/2}} q, \frac{q^{1/2}}{\alpha \beta w}, \frac{q^{1/2}}{\alpha \beta w} q; q^2 \right)_\infty}
\frac{1}{\left( \alpha^2 \beta^2; q^2 \right)_\infty} \notag \\ 
&=
\frac{1}{2}
\frac{\left( \frac{\gamma}{\alpha^2 \beta}; q^2 \right)_\infty}{\left( \alpha^2 \beta^2; q^2 \right)_\infty}
\frac{1}{\left( \frac{\gamma w}{\alpha q^{1/2}}, \frac{q^{1/2}}{\alpha \beta w}; q^2 \right)_\infty}
\frac{\left( q, q; q^2 \right)_\infty}{\left( \frac{\gamma w}{\alpha q^{1/2}} q, \frac{q^{1/2}}{\alpha \beta w} q; q^2 \right)_\infty}. 
\end{align}
Thus,
\begin{align} %
( \text{the left-hand side of \eqref{m1}} )
&=
\frac{1}{2}
\frac{\left( \frac{\gamma}{\alpha^2 \beta}; q^2 \right)_\infty}{\left( \alpha^2 \beta^2; q^2 \right)_\infty}
\frac{1}{\left( \frac{\gamma w}{\alpha q^{1/2}}, \frac{q^{1/2}}{\alpha \beta w}; q^2 \right)_\infty}
\frac{\left( q, q; q^2 \right)_\infty}{\left( \frac{\gamma w}{\alpha q^{1/2}} q, \frac{q^{1/2}}{\alpha \beta w} q; q^2 \right)_\infty} \notag \\ 
&\hspace{1em} \times
\left\{
\left( \frac{\beta \gamma w}{q^{1/2}}, \frac{q^{3/2}}{\beta \gamma w}, \frac{1}{\beta}, \beta q; q \right)_\infty
+
\left( \frac{\gamma w}{\alpha \beta q^{1/2}}, \frac{\alpha \beta q^{3/2}}{\gamma w}, \alpha \beta, \frac{q}{\alpha \beta}; q \right)_\infty
\right\}. 
\label{proof2}
\end{align}
As the next step, we would like to show the following relation:
\begin{cor} \label{co6} 
For any $w \in \Cbb^*$, we obtain
\begin{align} %
1
&=
\frac{1}{2}
\frac
{\left( q, q; q^2 \right)_\infty}
{\left( \frac{\alpha q^{1/2}}{\gamma w} q, \frac{\alpha \beta w}{q^{1/2}} q; q^2 \right)_\infty
\left( \frac{\gamma w}{\alpha q^{1/2}} q, \frac{q^{1/2}}{\alpha \beta w} q; q^2 \right)_\infty} \notag \\ 
&\hspace{1em} \times 
\left\{
\left( \frac{\beta \gamma w}{q^{1/2}}, \frac{q^{3/2}}{\beta \gamma w}, \frac{1}{\beta}, q \beta; q \right)_\infty
+
\left( \frac{\gamma w}{\alpha \beta q^{1/2}}, \frac{\alpha \beta q^{3/2}}{\gamma w}, \alpha \beta, \frac{q}{\alpha \beta}; q \right)_\infty
\right\}. \label{part2} 
\end{align}
\end{cor}

\begin{proof} 
The first line on the right-hand side of \eqref{part2} can be rewritten in terms of the theta function through Jacobi's triple product identity \eqref{jtpi} as
\begin{align} 
\frac{1}{2}
\frac
{\left( q, q; q^2 \right)_\infty}
{\left( \frac{\alpha q^{1/2}}{\gamma w} q, \frac{\alpha \beta w}{q^{1/2}} q; q^2 \right)_\infty
\left( \frac{\gamma w}{\alpha q^{1/2}} q, \frac{q^{1/2}}{\alpha \beta w} q; q^2 \right)_\infty}
&=
\frac{1}{2}
( q, q; q^2 )_\infty
\frac{( q^2; q^2 )_\infty}{\theta_{q^2} \left( \frac{\alpha q^{1/2}}{\gamma w} \right)}
\frac{( q^2; q^2 )_\infty}{\theta_{q^2} \left( \frac{\alpha \beta w}{q^{1/2}} \right)} \notag \\ 
&=
\frac{1}{2}
\frac
{( q, q; q )_\infty}
{\theta_{q^2} \left( \frac{\alpha q^{1/2}}{\gamma w} \right)
\theta_{q^2} \left( \frac{\alpha \beta w}{q^{1/2}} \right)} \notag \\ 
&=
\frac{1}{2}
\frac
{( q, q; q )_\infty}
{\theta_{q^2} \left( \frac{\alpha \beta}{q^{1/2} w} \right)
\theta_{q^2} \left( \frac{\alpha \beta w}{q^{1/2}} \right)} \notag \\ 
&=
\frac
{( q, q; q )_\infty}
{\theta_q \left( \frac{\alpha \beta}{q^{1/2}} \right)
\theta_q \left( \frac{w}{\alpha \beta^2} \right) 
+
\theta_q \left( w \right)
\theta_q \left( \frac{1}{\beta q^{1/2}} \right)}, 
\label{part12}
\end{align}
where we use $\beta \g = q$ and Corollary \ref{co5}. Similarly, Jacobi's triple product identity \eqref{jtpi} can also be applied to the second line of \eqref{part2} so that
\begin{align} 
&
\left\{
\left( \frac{\beta \gamma w}{q^{1/2}}, \frac{q^{3/2}}{\beta \gamma w}, \frac{1}{\beta}, \beta q; q \right)_\infty
+
\left( \frac{\gamma w}{\alpha \beta q^{1/2}}, \frac{\alpha \beta q^{3/2}}{\gamma w}, \alpha \beta, \frac{q}{\alpha \beta}; q \right)_\infty
\right\} \notag \\ 
&=
\frac{\theta_q \left( \frac{\beta \gamma}{q} w \right)}{( q; q )_\infty}
\frac{\theta_q \left( \frac{1}{\beta q^{1/2}} \right)}{( q; q )_\infty}
+
\frac{\theta_q \left( \frac{\gamma w}{\alpha \beta q} \right)}{( q; q )_\infty}
\frac{\theta_q \left( \frac{\alpha \beta}{q^{1/2}} \right)}{( q; q )_\infty} \notag \\ 
&=
\frac{1}{( q, q; q)_\infty}
\left\{
\theta_q \left( \frac{\beta \gamma}{q} w \right)
\theta_q \left( \frac{1}{\beta q^{1/2}} \right) 
+
\theta_q \left( \frac{\gamma w}{\alpha \beta q} \right)
\theta_q \left( \frac{\alpha \beta}{q^{1/2}} \right)
\right\} \notag \\ 
&=
\frac{1}{( q, q; q)_\infty}
\left\{
\theta_q \left( w \right)
\theta_q \left( \frac{1}{\beta q^{1/2}} \right) 
+
\theta_q \left( \frac{w}{\alpha \beta^2} \right)
\theta_q \left( \frac{\alpha \beta}{q^{1/2}} \right)
\right\}, 
\label{part22}
\end{align}
where we again insert the condition $\beta \g = q$ in the last equality. Combining \eqref{part12} and \eqref{part22} together, we obtain the relation \eqref{part2}.
\end{proof}
From Corollary \ref{co6}, we can immediately state the following identity:
\begin{cor} \label{co7} 
For any $w \in \Cbb^*$, we obtain
\begin{align*} %
\left( \frac{\alpha q^{1/2}}{\gamma w} q, \frac{\alpha \beta w}{q^{1/2}} q; q^2 \right)_\infty
&=
\frac{1}{2}
\frac{\left( q, q; q^2 \right)_\infty}{\left( \frac{\gamma w}{\alpha q^{1/2}} q, \frac{q^{1/2}}{\alpha \beta w} q; q^2 \right)_\infty} \\ 
&\hspace{1em} \times
\left\{
\left( \frac{\beta \gamma w}{q^{1/2}}, \frac{q^{3/2}}{\beta \gamma w}, \frac{1}{\beta}, q \beta; q \right)_\infty
+
\left( \frac{\gamma w}{\alpha \beta q^{1/2}}, \frac{\alpha \beta q^{3/2}}{\gamma w},\alpha \beta, \frac{q}{\alpha \beta}; q \right)_\infty
\right\}. 
\end{align*}
\end{cor}
Finally, substituting Corollary \ref{co7} into \eqref{proof2} results in
\begin{align*}
( \text{the left-hand side of \eqref{m1}} )
&=
\frac{\left( \frac{\gamma}{\alpha^2 \beta}; q^2 \right)_\infty}{\left( \alpha^2 \beta^2; q^2 \right)_\infty}
\frac{1}{\left( \frac{\gamma w}{\alpha q^{1/2}}, \frac{q^{1/2}}{\alpha \beta w}; q^2 \right)_\infty}
\left( \frac{\alpha q^{1/2}}{\gamma w} q, \frac{\alpha \beta w}{q^{1/2}} q; q^2 \right)_\infty \\ 
&=
( \text{the right-hand side of \eqref{m1}} ). 
\end{align*}
Therefore, Theorem \ref{mainth1} which we would like to show in the paper is proven.

\subsection{Proof of Theorem \ref{mainth2}}
Let us turn to implementing the proof of Theorem \ref{mainth2}. The basic process to prove it is the same of which we made use in the case of Theorem \ref{mainth1}. Hence, in the following, we omit some middle steps which are already written down in  the previous subsection.

Firstly, as for the path from \eqref{proof1} to \eqref{proof2}, we take Ramanujan's summation formula \eqref{rsum11} to re-xpress the left-hand side of \eqref{m2} as
\begin{align} \label{proof3} 
&
\frac{1}{2}
\frac{( q; q^2 )_\infty}{( q^2; q^2 )_\infty}
\left\{
\frac{\left( \frac{\gamma}{\alpha}; q \right)_\infty}{( \alpha \beta; q )_\infty}
{}_1\psi_1 \left( \alpha \beta; \frac{\gamma}{\alpha}; q, \frac{\gamma w}{\alpha q^{1/2}} \right)
-
\frac{( \beta^2 \gamma; q )_\infty}{( \frac{1}{\beta}; q )_\infty}
{}_1\psi_1 \left( \frac{1}{\beta}; \beta^2 \gamma; q, \frac{\gamma w}{\alpha q^{1/2}} \right)
\right\} \notag \\ 
&=
\frac{1}{2}
\frac{\left( \frac{\gamma}{\alpha^2 \beta}; q^2 \right)_\infty}{\left( \alpha^2 \beta^2; q^2 \right)_\infty}
\frac{1}{\left( \frac{\gamma w}{\alpha q^{1/2}} q, \frac{q^{1/2}}{\alpha \beta w} q; q^2 \right)_\infty}
\frac{\left( q, q; q^2 \right)_\infty}{\left( \frac{\gamma w}{\alpha q^{1/2}}, \frac{q^{1/2}}{\alpha \beta w}; q^2 \right)_\infty} \notag \\ 
&\hspace{1em} \times
\left\{
\left( \frac{\beta \gamma w}{q^{1/2}}, \frac{q^{3/2}}{\beta \gamma w}, \frac{1}{\beta}, \beta q; q \right)_\infty
-
\left( \frac{\gamma w}{\alpha \beta q^{1/2}}, \frac{\alpha \beta q^{3/2}}{\gamma w}, \alpha \beta, \frac{q}{\alpha \beta}; q \right)_\infty
\right\}. 
\end{align}

Secondly, we would like to show the following relation:
\begin{cor} \label{coro1} 
For any $w \in \Cbb^*$, we obtain
\begin{align} %
\alpha \beta
&=
\frac{1}{2}
\frac
{\left( q, q; q^2 \right)_\infty}
{\left( \frac{\alpha q^{1/2}}{\g w} q^2, \frac{\alpha \beta w}{q^{1/2}} q^2; q^2 \right)_\infty
\left( \frac{\g w}{\alpha q^{1/2}}, \frac{q^{1/2}}{\alpha \beta w}; q^2 \right)_\infty} \notag \\ 
&\hspace{1em} \times 
\left\{
\left( \frac{\beta \g w}{q^{1/2}}, \frac{q^{3/2}}{\beta \g w}, \frac{1}{\beta}, \beta q; q \right)_\infty
-
\left( \frac{\g w}{\alpha \beta q^{1/2}}, \frac{\alpha \beta q^{3/2}}{\gamma w}, \alpha \beta, \frac{q}{\alpha \beta}; q \right)_\infty
\right\}. \label{part2b} 
\end{align}
\end{cor}

\begin{proof} 
We again rely on Jacobi's triple product identity \eqref{jtpi} to rewrite the right-hand side of \eqref{part2b}. Its first line  is given by means of the theta function with base $q$ as follows:
\begin{align} 
\frac{1}{2}
\frac
{\left( q, q; q^2 \right)_\infty}
{\left( \frac{\alpha q^{1/2}}{\gamma w} q^2, \frac{\alpha \beta w}{q^{1/2}} q^2; q^2 \right)_\infty
\left( \frac{\gamma w}{\alpha q^{1/2}}, \frac{q^{1/2}}{\alpha \beta w}; q^2 \right)_\infty}
&=
\frac{1}{2}
( q, q; q^2 )_\infty
\frac{( q^2; q^2 )_\infty}{\theta_{q^2} \left( \frac{\alpha q^{1/2}}{\gamma w} q \right)}
\frac{( q^2; q^2 )_\infty}{\theta_{q^2} \left( \frac{\alpha \beta w}{q^{1/2}} q \right)} \notag \\ 
&=
\frac{1}{2}
\frac
{( q, q; q )_\infty}
{\theta_{q^2} \left( \frac{\alpha \beta}{q^{1/2} w} q \right)
\theta_{q^2} \left( \frac{\alpha \beta w}{q^{1/2}} q \right)} \notag \\ 
&=
\alpha \beta
\frac
{( q, q; q )_\infty}
{\theta_q \left( w \right)
\theta_q \left( \frac{1}{\beta q^{1/2}} \right) 
-
\theta_q \left( \frac{w}{\alpha \beta^2} \right)
\theta_q \left( \frac{\alpha \beta}{q^{1/2}} \right)}, 
\label{part12b}
\end{align}
where we use $\beta \g = q$ and Corollary \ref{corol1}. Also, the second line of \eqref{part2b} become the linear combination of the products of the theta function thanks to Jacobi's triple product identity \eqref{jtpi},
\begin{align} 
&
\left\{
\left( \frac{\beta \gamma w}{q^{1/2}}, \frac{q^{3/2}}{\beta \gamma w}, \frac{1}{\beta}, \beta q; q \right)_\infty
-
\left( \frac{\gamma w}{\alpha \beta q^{1/2}}, \frac{\alpha \beta q^{3/2}}{\gamma w}, \alpha \beta, \frac{q}{\alpha \beta}; q \right)_\infty
\right\} \notag \\ 
&=
\frac{1}{( q, q; q)_\infty}
\left\{
\theta_q \left( w \right)
\theta_q \left( \frac{1}{\beta q^{1/2}} \right) 
-
\theta_q \left( \frac{w}{\alpha \beta^2} \right)
\theta_q \left( \frac{\alpha \beta}{q^{1/2}} \right)
\right\}, 
\label{part22b}
\end{align}
where we again put the condition $\beta \g = q$. Combining \eqref{part12b} and \eqref{part22b} gives us the relation \eqref{part2b}.
\end{proof}
In addition, we can claim the following identity from Corollary \ref{coro1}:
\begin{cor} \label{coro2} 
For any $w \in \Cbb^*$, we have
\begin{align*} %
\alpha \beta
\left( \frac{\alpha q^{1/2}}{\gamma w} q^2, \frac{\alpha \beta w}{q^{1/2}} q^2; q^2 \right)_\infty
&=
\frac{1}{2}
\frac{\left( q, q; q^2 \right)_\infty}{\left( \frac{\gamma w}{\alpha q^{1/2}} q, \frac{q^{1/2}}{\alpha \beta w} q; q^2 \right)_\infty} \\ 
&\hspace{1em} \times
\left\{
\left( \frac{\beta \gamma w}{q^{1/2}}, \frac{q^{3/2}}{\beta \gamma w}, \frac{1}{\beta}, q \beta; q \right)_\infty
-
\left( \frac{\gamma w}{\alpha \beta q^{1/2}}, \frac{\alpha \beta q^{3/2}}{\gamma w},\alpha \beta, \frac{q}{\alpha \beta}; q \right)_\infty
\right\}. 
\end{align*}
\end{cor}

Finally, we insert the relation of Corollary \ref{coro2} into \eqref{proof3} which is the left-hand side of \eqref{m2},
\begin{align*}
( \text{the left-hand side of \eqref{m2}} )
&=
\alpha \beta
\frac{\left( \frac{\g}{\alpha^2 \beta}; q^2 \right)_\infty}{\left( \alpha^2 \beta^2; q^2 \right)_\infty}
\frac{1}{\left( \frac{\g w}{\alpha q^{1/2}} q, \frac{q^{1/2}}{\alpha \beta w} q; q^2 \right)_\infty}
\left( \frac{\alpha q^{1/2}}{\gamma w} q^2, \frac{\alpha \beta w}{q^{1/2}} q^2; q^2 \right)_\infty \\ 
&=
( \text{the right-hand side of \eqref{m2}} ). 
\end{align*}
Consequently, we can provide the complete proof of Theorem \ref{mainth2}.

\subsection{Application to physics}
Let us consider the special cases of Theorem \ref{mainth1} and \ref{mainth2}. We obtain the following formulae by setting $\alpha = - a$, $\beta = - a^{- \frac{1}{2}}$ and $\gamma = - a^{\frac{1}{2}} q$: 
\begin{cor} \label{cormain1} 
For any $w \in \mathbb{C}^*$ such that $|q/a| < |q^{1/2} w/a^{1/2}| < 1$, we have
\begin{align*}
&
\frac{1}{2}
\frac{( q; q^2 )_\infty}{( q^2; q^2 )_\infty} 
\left\{
\frac{( a^{- \frac{1}{2}} q; q )_\infty}{( a^{\frac{1}{2}}; q )_\infty}
{}_1\psi_1 \left( a^{\frac{1}{2}}; a^{- \frac{1}{2}} q; q, q^{\frac{1}{2}} a^{- \frac{1}{2}} w \right)
+
\frac{( - a^{- \frac{1}{2}} q; q )_\infty}{( - a^{\frac{1}{2}}; q )_\infty}
{}_1\psi_1 \left( - a^{\frac{1}{2}}; - a^{- \frac{1}{2}} q; q, q^{\frac{1}{2}} a^{- \frac{1}{2}} w \right)
\right\} \\ 
&=
\frac
{( \tilde{a}^{-1} q, \tilde{a}^{\frac{1}{2}} \tilde{w}^{- 1} q^{\frac{1}{2}}, \tilde{a}^{\frac{1}{2}} \tilde{w} q^{\frac{1}{2}}; q^2 )_\infty}
{( \tilde{a}, \tilde{a}^{- \frac{1}{2}} \tilde{w} q^{\frac{1}{2}}, \tilde{a}^{- \frac{1}{2}} \tilde{w}^{- 1} q^{\frac{1}{2}}; q^2 )_\infty}, 
\end{align*}
where $a = \tilde{a}$ and $w = \tilde{w}$. 
\end{cor}

\begin{cor} \label{cormain2} 
For any $w \in \mathbb{C}^*$ such that $|q/a| < |q^{3/2} w/a^{1/2}| < 1$, we have
\begin{align*}
&
\frac{1}{2}
\frac{( q; q^2 )_\infty}{( q^2; q^2 )_\infty} 
\left\{
\frac{( a^{- \frac{1}{2}} q; q )_\infty}{( a^{\frac{1}{2}}; q )_\infty}
{}_1\psi_1 \left( a^{\frac{1}{2}}; a^{- \frac{1}{2}} q; q, q^{\frac{1}{2}} a^{- \frac{1}{2}} w \right)
-
\frac{( - a^{- \frac{1}{2}} q; q )_\infty}{( - a^{\frac{1}{2}}; q )_\infty}
{}_1\psi_1 \left( - a^{\frac{1}{2}}; - a^{- \frac{1}{2}} q; q, q^{\frac{1}{2}} a^{- \frac{1}{2}} w \right)
\right\} \\ 
&=
\tilde{a}^{\frac{1}{2}}
\frac
{( \tilde{a}^{-1} q, \tilde{a}^{\frac{1}{2}} \tilde{w}^{- 1} q^{\frac{3}{2}}, \tilde{a}^{\frac{1}{2}} \tilde{w} q^{\frac{3}{2}}; q^2 )_\infty}
{( \tilde{a}, \tilde{a}^{- \frac{1}{2}} \tilde{w} q^{\frac{3}{2}}, \tilde{a}^{- \frac{1}{2}} \tilde{w}^{- 1} q^{\frac{3}{2}}; q^2 )_\infty}, 
\end{align*}
where $a = \tilde{a}$ and $w = \tilde{w}$. 
\end{cor}
Those are the identities firstly found from the physics viewpoint to explore the simplest version of Abelian mirror symmetry on $\mathbf{R}P^2 \times S^1$ \cite{MMT}. This symmetry states that two distinct supersymmetric gauge theories named the supersymmetric quantum electrodynamics (SQED) and the XYZ-model are physically equivalent. The left-hand and right-hand sides of Corollary \ref{cormain1} and \ref{cormain2} are the exact formulae of the superconformal indices of the SQED and the XYZ-model, respectively. $a$ and $w$ (resp. $\tilde{a}$ and $\tilde{w}$) are weight parameters called fugacities associated to U$(1)$ symmetries in the SQED (resp. the XYZ-model). The conditions $a = \tilde{a}$ and $w = \tilde{w}$ represents the correspondence of each U$(1)$ symmetry of two theories, which is expected from physical arguments. Therefore, Theorem \ref{mainth1} and \ref{mainth2} are the generalized formulae motivated by the physical results \cite{MMT}, and those are interesting examples that purely physical discussions lead to mathematically nontrivial results.

\section{The $q \to 1 - 0$ limit of the new formulae} \label{Limit}
In this section, we would like to show the $q \to 1 - 0$ limit of our new formulae
\begin{align} 
&
\frac{1}{2}
\left\{
\frac{\left( \frac{\gamma}{\alpha}; q \right)_\infty}{( \alpha \beta; q )_\infty}
{}_1\psi_1 \left( \alpha \beta; \frac{\gamma}{\alpha}; q, \frac{\gamma w}{\alpha q^{1/2}} \right)
+
\frac{( \beta^2 \gamma; q )_\infty}{( \frac{1}{\beta}; q )_\infty}
{}_1\psi_1 \left( \frac{1}{\beta}; \beta^2 \gamma; q, \frac{\gamma w}{\alpha q^{1/2}} \right)
\right\} \notag \\ 
&=
\frac{( q^2; q^2 )_\infty}{( q; q^2 )_\infty}
\frac
{\left( \frac{\gamma}{\alpha^2 \beta}; q^2 \right)_\infty}
{( \alpha^2 \beta^2; q^2 )}
\frac
{\left( \frac{\alpha q^{1/2}}{\gamma w} q, \frac{\alpha \beta w}{q^{1/2}} q; q^2 \right)_\infty}
{\left( \frac{\gamma w}{\alpha q^{1/2}}, \frac{q^{1/2}}{\alpha \beta w}; q^2 \right)_\infty},
\label{col}
\end{align}
and
\begin{align} 
&
\frac{1}{2}
\left\{
\frac{\left( \frac{\gamma}{\alpha}; q \right)_\infty}{( \alpha \beta; q )_\infty}
{}_1\psi_1 \left( \alpha \beta; \frac{\gamma}{\alpha}; q, \frac{\gamma w}{\alpha q^{1/2}} \right)
-
\frac{( \beta^2 \gamma; q )_\infty}{( \frac{1}{\beta}; q )_\infty}
{}_1\psi_1 \left( \frac{1}{\beta}; \beta^2 \gamma; q, \frac{\gamma w}{\alpha q^{1/2}} \right)
\right\} \notag \\ 
&=
\alpha \beta
\frac{( q^2; q^2 )_\infty}{( q; q^2 )_\infty}
\frac
{\left( \frac{\gamma}{\alpha^2 \beta}; q^2 \right)_\infty}
{( \alpha^2 \beta^2; q^2 )}
\frac
{\left( \frac{\alpha q^{1/2}}{\gamma w} q^2, \frac{\alpha \beta w}{q^{1/2}} q^2; q^2 \right)_\infty}
{\left( \frac{\gamma w q^{1/2}}{\alpha}, \frac{q^{3/2}}{\alpha \beta w}; q^2 \right)_\infty}, 
\label{mcor22}
\end{align}
where $\alpha \beta^2 = - 1$ and $\beta \gamma = q$.  At first, we set $\alpha = q^a$, $\beta = q^b$ and $\gamma = q^c$,  then we have $q^{a + b + 1 - c} = - 1$. We also introduce the weight function
\begin{align} 
W ( a, b, c; q )
:=
( 1 - q^2 )^{- \frac{4 a + 3 b - c - 1}{2}},
\label{weight}
\end{align}
where $a$ is constrained by $a = - 2b$. This condition and relation $b + c = 1$ which is equivalent to $\beta \g = q$ imply
\begin{align}
- \frac{4 a + 3 b - c - 1}{2} = 3 b + c.
\label{rellim}
\end{align}
The aim of this section is to show the following theorem as the $q \to 1 - 0$ limit of the formulae \eqref{col} and \eqref{mcor22} multiplied by the weight function $W ( a, b, c; q )$:
\begin{thm} \label{mainlim2} 
For any $w \in \mathbb{C}^*$, we have
\begin{align*}
\frac{2^{2 b + 1}}{\G ( b + 1 )}
{}_1H_1 \left( - b; b + 1; w \right)
=
\frac
{\G \lp \frac{1}{2} \rp}
{\G \lp \frac{2 b + 1}{2} \rp}
( - w )^{- b}
( 1 - w )^{2 b}.
\end{align*}
\end{thm}

\begin{proof}
Let us start with considering the left-hand side of \eqref{col} multiplied by the weight function \eqref{weight}.
\begin{enumerate} 
\item[(i)] The first term on the left-hand side of \eqref{col} can be expressed in terms of the $q$-gamma function,
\begin{align} 
&
( 1 - q^2 )^{3 b + c}
\frac{\left( q^{c - a}; q \right)_\infty}{( q^{a + b}; q )_\infty}
{}_1\psi_1 \left( q^{a + b}; q^{c - a}; q, q^{c - a - \frac{1}{2}} w \right) \notag \\ 
&=
(1 + q)^{3 b + c}
\cdot
( 1 - q )^{c - a - 1}
\frac{\left( q^{c - a}; q \right)_\infty}{( q; q )_\infty}
\cdot
( 1 - q )^{1 - ( a + b )}
\frac{( q; q )_\infty}{( q^{a + b}; q )_\infty}
\cdot
{}_1\psi_1 \left( q^{a + b}; q^{c - a}; q, q^{c - a - \frac{1}{2}} w \right) \notag \\ 
&=
(1 + q)^{3 b + c}
\frac{\G_{q} ( a + b )}{\G_{q} ( c - a )}
{}_1\psi_1 \left( q^{a + b}; q^{c - a}; q, q^{c - a - \frac{1}{2}} w \right) \notag \\ 
&=
(1 + q)^{2 b + 1}
\frac{\G_{q} ( - b )}{\G_{q} ( b + 1 )}
{}_1\psi_1 \left( q^{- b}; q^{b + 1}; q, q^{b + \frac{1}{2}} w \right). \notag 
\end{align}
The $q \to 1 - 0$ limit of this part becomes
\begin{align}
2^{2 b + 1}
\frac{\G ( - b )}{\G ( b + 1 )}
{}_1H_1 \left( - b; b + 1; w \right).
\label{limitp2}
\end{align}

\item[(ii)] The second term on the left-hand side of \eqref{col} is rewritten in the same manner as the first term,
\begin{align} 
&
( 1 - q^2 )^{3 b + c}
\frac{( q^{2 b + c}; q )_\infty}{( q^{- b}; q )_\infty}
{}_1\psi_1 \left( q^{- b}; q^{2 b + c}; q, q^{c - a - \frac{1}{2}} w \right) \notag \\ 
&=
(1 + q)^{3 b + c}
\frac{\G_{q} ( - b )}{\G_{q} ( 2 b + c )}
{}_1\psi_1 \left( q^{- b}; q^{2 b + c}; q, q^{c - a - \frac{1}{2}} w \right) \notag \\ 
&=
(1 + q)^{2 b + 1}
\frac{\G_{q} ( - b )}{\G_{q} ( b + 1 )}
{}_1\psi_1 \left( q^{- b}; q^{b + 1}; q, q^{c - a - \frac{1}{2}} w \right), \notag 
\end{align}
namely, the $q \to 1 - 0$ limit of this part also produces the contribution \eqref{limitp2}.
\end{enumerate}
As a result, we can take the $q \to 1 - 0$ limit of the left-hand side of \eqref{col} as
\begin{align} 
( \text{the left-hand side of \eqref{col}} )
&\to
\frac{1}{2}
\cdot
2
\cdot
2^{2 b + 1}
\frac{\G ( - b )}{\G ( b + 1 )}
{}_1H_1 \left( - b; b + 1; w \right) \notag \\ 
&=
2^{2 b + 1}
\frac{\G ( - b )}{\G ( b + 1 )}
{}_1H_1 \left( - b; b + 1; w \right). 
\label{limitp1}
\end{align}

On the other hand, the right-hand side multiplied by the weight function \eqref{weight} can be rewritten by using, in addition to $a = - 2b$ and $b + c = 1$, the $q$-gamma function and the theta function with base $q^2$ as follows:
\begin{align} 
&
W ( a, b, c; q )
\frac
{\left( q^2; q^2 \right)_\infty}
{\lp q; q^2 \rp}
\frac
{\left( q^{c - 2 a - b}; q^2 \right)_\infty}
{\lp q^{2 ( a + b )}; q^2 \rp}
\frac
{\left( q^{a - c + \frac{3}{2}} w^{- 1}, q^{a + b + \frac{1}{2}} w; q^2 \right)_\infty}
{\left( q^{c - a - \frac{1}{2}} w, q^{- a - b + \frac{1}{2}} w^{- 1}; q^2 \right)_\infty} \notag \\ 
&=
( 1 - q^2 )^{\frac{1}{2}}
\frac
{\left( q^2; q^2 \right)_\infty}
{\lp ( q^2 )^{\frac{1}{2}}; q^2 \rp}
\cdot
( 1 - q^2 )^{\frac{1}{2} ( c - 2 a - b ) - 1}
\frac
{\left( ( q^2 )^{\frac{1}{2} ( c - 2 a - b )}; q^2 \right)_\infty}
{\lp q^2; q^2 \rp}
\cdot
( 1 - q^2 )^{1 - ( a + b )}
\frac
{\left( q^2; q^2 \right)_\infty}
{\lp ( q^2 )^{a + b}; q^2 \rp}
\notag \\ 
&\hspace{1em} \times
\lc
\frac
{\left( q^{a - c + \frac{3}{2}} w^{- 1}; q^2 \right)_\infty}
{\left( q^{- a - b + \frac{1}{2}} w^{- 1}; q^2 \right)_\infty}
\cdot
\frac
{\lp q^{- a + c + \frac{1}{2}} w, q^2; q^2 \rp_\infty}
{\lp q^{a + b + \frac{3}{2}} w, q^2; q^2 \rp_\infty}
\rc
\cdot
\frac
{\left( q^{a + b + \frac{1}{2}} w; q^2 \right)_\infty}
{\left( q^{c - a - \frac{1}{2}} w; q^2 \right)_\infty}
\cdot
\frac
{\lp q^{a + b + \frac{3}{2}} w, q^2; q^2 \rp_\infty}
{\lp q^{- a + c + \frac{1}{2}} w, q^2; q^2 \rp_\infty} \notag \\ 
&=
\frac
{\G_{q^2} \lp \frac{1}{2} \rp \G_{q^2} \lp a + b \rp}
{\G_{q^2} \lp \frac{c - 2 a - b}{2} \rp}
\lc
\frac
{\theta_{q^2} \lp q^{- a + c - \frac{1}{2}} w \rp}
{\theta_{q^2} \lp q^{a + b + \frac{1}{2}} w \rp}
\rc \notag \\ 
&\hspace{1em} \times
\frac
{\left( q^{a + b + \frac{1}{2}} w; q^2 \right)_\infty}
{\left( w; q^2 \right)_\infty}
\frac
{\left( w; q^2 \right)_\infty}
{\left( q^{c - a - \frac{1}{2}} w; q^2 \right)_\infty}
\cdot
\frac
{\lp q^{a + b + \frac{3}{2}} w, q^2; q^2 \rp_\infty}
{\lp w; q^2 \rp_\infty}
\frac
{\lp w; q^2 \rp_\infty}
{\lp q^{- a + c + \frac{1}{2}} w, q^2; q^2 \rp_\infty} \notag \\ 
&=
\frac
{\G_{q^2} \lp \frac{1}{2} \rp \G_{q^2} \lp - b \rp}
{\G_{q^2} \lp \frac{2 b + 1}{2} \rp}
\frac
{\theta_{q^2} \lp q^{b + \frac{1}{2}} w \rp}
{\theta_{q^2} \lp q^{- b + \frac{1}{2}} w \rp} \notag \\ 
&\hspace{1em} \times
\frac
{\left( q^{- b + \frac{1}{2}} w; q^2 \right)_\infty}
{\left( w; q^2 \right)_\infty}
\frac
{\left( w; q^2 \right)_\infty}
{\left( q^{b + \frac{1}{2}} w; q^2 \right)_\infty}
\cdot
\frac
{\lp q^{- b + \frac{3}{2}} w; q^2 \rp_\infty}
{\lp w; q^2 \rp_\infty}
\frac
{\lp w; q^2 \rp_\infty}
{\lp q^{b + \frac{3}{2}} w; q^2 \rp_\infty}. 
\label{limitp4}
\end{align}
As $q \to 1 - 0$, the limiting formulae \eqref{limgamma}, \eqref{limt1}, and \eqref{limbin} reduce the function \eqref{limitp4} to 
\begin{align} 
( \text{the right-hand side of \eqref{col}} )
&\to
\frac
{\G \lp \frac{1}{2} \rp \G \lp - b \rp}
{\G \lp \frac{2 b + 1}{2} \rp}
( - w )^{- b}
( 1 - w )^{\frac{1}{2} \lc \lp b - \frac{1}{2} \rp + \lp b + \frac{1}{2} \rp + \lp b - \frac{3}{2} \rp + \lp b + \frac{3}{2} \rp \rc} \notag \\ 
&=
\frac
{\G \lp \frac{1}{2} \rp \G \lp - b \rp}
{\G \lp \frac{2 b + 1}{2} \rp}
( - w )^{- b}
( 1 - w )^{2 b}. 
\label{limitp5}
\end{align}
Combining the results \eqref{limitp1} and \eqref{limitp5}, we obtain Theorem \ref{mainlim2} as a conclusion.
\end{proof}

Further, we can conclude that \eqref{mcor22} in the $q \to 1 - 0$ limit leads to Theorem \ref{mainlim2}.
\begin{proof}
First of all, we rely on $\alpha \beta^2 = - 1$ on the left-hand side of \eqref{mcor22} such that
\begin{align} 
&
\frac{1}{2}
\left\{
\frac{\left( \frac{\gamma}{\alpha}; q \right)_\infty}{( \alpha \beta; q )_\infty}
{}_1\psi_1 \left( \alpha \beta; \frac{\gamma}{\alpha}; q, \frac{\gamma w}{\alpha q^{1/2}} \right)
-
\frac{( \beta^2 \gamma; q )_\infty}{( \frac{1}{\beta}; q )_\infty}
{}_1\psi_1 \left( \frac{1}{\beta}; \beta^2 \gamma; q, \frac{\gamma w}{\alpha q^{1/2}} \right)
\right\} \notag \\ 
&=
\frac{1}{2}
\left\{
\frac{\left( \frac{\gamma}{\alpha}; q \right)_\infty}{( \alpha \beta; q )_\infty}
{}_1\psi_1 \left( \alpha \beta; \frac{\gamma}{\alpha}; q, \frac{\gamma w}{\alpha q^{1/2}} \right)
+
\alpha \beta^2
\frac{( \beta^2 \gamma; q )_\infty}{( \frac{1}{\beta}; q )_\infty}
{}_1\psi_1 \left( \frac{1}{\beta}; \beta^2 \gamma; q, \frac{\gamma w}{\alpha q^{1/2}} \right)
\right\}. 
\label{limitn1}
\end{align}
Because the only difference of \eqref{limitn1} from the left-hand side of \eqref{col} is the presence of factor $\alpha \beta^{2} = q^{a + 2 b}$ in the second term, we can deform \eqref{limitn1} multiplied by the weight function \eqref{weight} in the same way as to obtain \eqref{limitp1} in the $q \to 1 - 0$ limit, thus,
\begin{align} 
( \text{the left-hand side of \eqref{mcor22}} )
\to
2^{2 b + 1}
\frac{\G ( - b )}{\G ( b + 1 )}
{}_1H_1 \left( - b; b + 1; w \right).
\label{limitn2}
\end{align}

Let us turn to the right-hand side of \eqref{mcor22} multiplied by the weight function \eqref{weight},
\begin{align} 
&
W ( a, b, c; q )
q^{a + b}
\frac
{\left( q^2, q^{c - 2 a - b}; q^2 \right)_\infty}
{\lp q, q^{2 ( a + b )}; q^2 \rp}
\frac
{\left( q^{a - c + \frac{5}{2}} w^{- 1}, q^{a + b + \frac{3}{2}} w; q^2 \right)_\infty}
{\left( q^{c - a + \frac{1}{2}} w, q^{- a - b + \frac{3}{2}} w^{- 1}; q^2 \right)_\infty} \notag \\ 
&=
q^{a + b}
\frac
{\G_{q^2} \lp \frac{1}{2} \rp \G_{q^2} \lp a + b \rp}
{\G_{q^2} \lp \frac{c - 2 a - b}{2} \rp}
\frac
{\theta_{q^2} \lp q^{- a + c - \frac{3}{2}} w \rp}
{\theta_{q^2} \lp q^{a + b - \frac{1}{2}} w \rp} \notag \\ 
&\hspace{1em} \times
\frac
{\left( q^{a + b + \frac{3}{2}} w; q^2 \right)_\infty}
{\left( w; q^2 \right)_\infty}
\frac
{\left( w; q^2 \right)_\infty}
{\left( q^{c - a + \frac{1}{2}} w; q^2 \right)_\infty}
\cdot
\frac
{\lp q^{a + b + \frac{1}{2}} w, q^2; q^2 \rp_\infty}
{\lp w; q^2 \rp_\infty}
\frac
{\lp w; q^2 \rp_\infty}
{\lp q^{- a + c - \frac{1}{2}} w, q^2; q^2 \rp_\infty} \notag \\ 
&=
q^{a + b}
\frac
{\G_{q^2} \lp \frac{1}{2} \rp \G_{q^2} \lp - b \rp}
{\G_{q^2} \lp \frac{2 b + 1}{2} \rp}
\frac
{\theta_{q^2} \lp q^{b - \frac{1}{2}} w \rp}
{\theta_{q^2} \lp q^{- b - \frac{1}{2}} w \rp} \notag \\ 
&\hspace{1em} \times
\frac
{\left( q^{- b + \frac{3}{2}} w; q^2 \right)_\infty}
{\left( w; q^2 \right)_\infty}
\frac
{\left( w; q^2 \right)_\infty}
{\left( q^{b + \frac{3}{2}} w; q^2 \right)_\infty}
\cdot
\frac
{\lp q^{- b + \frac{1}{2}} w; q^2 \rp_\infty}
{\lp w; q^2 \rp_\infty}
\frac
{\lp w; q^2 \rp_\infty}
{\lp q^{b + \frac{1}{2}} w; q^2 \rp_\infty}. 
\label{limitn4}
\end{align}
Again, the limiting formulae \eqref{limgamma}, \eqref{limt1}, and \eqref{limbin} provide
\begin{align} 
( \text{the right-hand side of \eqref{mcor22}} )
\to
\frac
{\G \lp \frac{1}{2} \rp \G \lp - b \rp}
{\G \lp \frac{2 b + 1}{2} \rp}
( - w )^{- b}
( 1 - w )^{2 b}.
\label{limitn5}
\end{align}
The expressions \eqref{limitn2} and \eqref{limitn5} are completely identical with \eqref{limitp1} and \eqref{limitp5}, respectively. Therefore, the $q \to 1 - 0$ limit of \eqref{mcor22} also results in Theorem \ref{mainlim2}.
\end{proof}

Finally, we would like to comment on the physical perspective of Theorem \ref{mainlim2}. As we mentioned, the new formula shown in this paper strongly supports Abelian mirror symmetry on $\mathbf{R}P^2 \times S^1$. The $q \to 1 - 0$ limit corresponds physically to the limit where the radius of $S^1$ is taken to be zero, and thus we expect that Theorem \ref{mainlim2} may lead to the certain relationship of partition functions on $\mathbf{R}P^2$. We hope to provide the physical meaning to this purely mathematical conclusion as a future problem.

\section*{Acknowledgments}
The work of H.M. was supported in part by the JSPS Research Fellowship for Young Scientists.

\providecommand{\href}[2]{#2}\begingroup\raggedright\endgroup

\end{document}